\documentclass[a4paper,12pt]{article}
\usepackage[utf8]{inputenc}
\usepackage{amsfonts}
\usepackage{latexsym}
\usepackage{amsmath}
\usepackage{amssymb}
\usepackage{graphicx}

\usepackage{mathrsfs}

\usepackage{bm}
\usepackage{tikz}

\usepackage{fancybox}
\usepackage[hidelinks]{hyperref}

\newcommand{\C}{\mathbb{C}}
\newcommand{\N}{\mathbb{N}}

\newtheorem{defin}{Definition}[section]

\newtheorem{theorem}[defin]{Theorem}
\newtheorem{proposition}[defin]{Proposition}
\newtheorem{exa}[defin]{Example}
\newenvironment{example}{\begin{exa}\rm}{\end{exa}}
\newtheorem{lemma}[defin]{Lemma}
\newtheorem{corollary}[defin]{Corollary}
\newenvironment{proof}
{\noindent{\it Proof.}}{\hfill $\Box$\par\vspace{2.5mm}}
\newenvironment{remark}
{\par\vspace{2.5mm}\noindent{\bf Remark.}}{\par\vspace{2.5mm}}

\newtheorem{que}{Question}

\newtheorem{pro}{Problem}

\numberwithin{equation}{section}

\title{\bf\Large Completely regular growth of solutions to linear differential equations with exponential polynomial coefficients}
\author{Xing-Yu Li}
\date{}
\begin{document}
	\maketitle
	
	\begin{abstract}
Consider the linear differential equation
\[
 f^{(n)}+A_{n-1}f^{(n-1)}+\cdots+A_0f=0,
\]
where the coefficients are exponential polynomials.
We prove that every finite-order transcendental entire solution has completely regular growth and obtain asymptotic formulas for \(\log |f|\) on almost every ray. This answers a question of Heittokangas et al.\ \cite[p.~33]{HITW2022}.
The proof combines logarithmic derivative estimates with asymptotic
integration along rays.
We express the angular densities of zeros in terms of the indicator.
We also extend the result beyond exponential polynomial coefficients under a sectorial Laurent–Puiseux condition.

\bigskip
\noindent
\textbf{Keywords:} Linear differential equations, exponential polynomials,
entire solutions of finite order, completely regular growth,
asymptotic integration.

\medskip
\noindent
\textbf{2020 MSC:} Primary 34M10; Secondary 30D15, 34M30.
\end{abstract}

\section{Introduction}
Complex oscillation and the growth of solutions of linear differential
equations with entire coefficients have been studied extensively.
Exponential polynomials play an important role in oscillation and
zero distribution; see the survey by
Heittokangas, Ishizaki, Tohge and Wen \cite{HITW2022}.

\bigskip
An \emph{exponential polynomial} is an entire function of the form
\begin{equation}\label{exppoly.eq1}
 G=P_1e^{Q_1}+\cdots+P_le^{Q_l},
\end{equation}
where $P_j,Q_j$ are polynomials and $l\in\mathbb N$.
An \emph{exponential sum} is an exponential polynomial of the form
\begin{equation}\label{expsum.eq}
	G(z)=F_1(z)e^{\lambda_1z}+\cdots+F_m(z)e^{\lambda_mz},
\end{equation}
where $m\in\mathbb N$, the $F_j$ are polynomials, and
$\lambda_1,\ldots,\lambda_m\in\C$ are distinct constants.
The $\lambda_j$ are called \emph{frequencies} (or \emph{leading coefficients}),
and the $F_j$ are called \emph{multipliers}.

\bigskip
For an entire function $f$ of finite positive order $\rho(f)$,
the Phragm\'en--Lindel\"of indicator is
\[
 h_f(\theta)=\limsup_{r\to\infty}
             r^{-\rho(f)}\log|f(re^{i\theta})|,
 \qquad 0\le\theta<2\pi.
\]
Throughout the paper, completely regular growth is taken with respect
to the constant proximate order $\rho(r)\equiv\rho(f)$, and finite
nonzero type is required.
We say that $f$ has \emph{completely regular growth} (\textit{c.r.g.}) if there
are Euclidean discs $D(z_k,r_k)$ satisfying
\begin{equation}\label{r.eq}
 \sum_{|z_k|\le r}r_k=o(r)
\end{equation}
such that
\begin{equation}\label{crg}
 \log|f(re^{i\theta})|=(h_f(\theta)+o(1))r^{\rho(f)},
 \qquad re^{i\theta}\notin\bigcup_kD(z_k,r_k),
\end{equation}
as $r\to\infty$, uniformly in $\theta$;
see \cite[pp.~139--140]{Levin1} or \cite[pp.~6--8]{Ronkin}.
     
Every transcendental exponential polynomial is of \textit{c.r.g.}; see
\cite[Lemma~1.3]{GOP}.
Petrenko \cite[pp.~104--112]{Petru} and Steinmetz \cite{Stein}
independently proved that transcendental entire solutions of linear
differential equations with polynomial coefficients are of \textit{c.r.g}.

When at least one entire coefficient is transcendental, a theorem of
Frei \cite{Frei} shows that there are at most $n-1$ linearly independent
solutions of finite order.
      Gol'dberg and Ostrovski\v{i} stated a classical problem in the  book \cite[p.~300]{problembook}:
    
    \bigskip
     \noindent
    	\textbf{(Gol'dberg-Ostrovski\v{i})} \emph{Suppose that $f$ is a finite order entire transcendental solution of 
    		$$
    	f^{(n)}+a_{n-1}f^{(n-1)}+\cdots+a_1f'+a_0f=0,
    	$$
    	whose coefficients are entire functions of \textit{c.r.g}. Is $f$ also of \textit{c.r.g.}?}

    \bigskip
Several positive results for second-order equations are discussed in 
\cite[Section~4]{GOP}. In particular, they recall cases with coefficients
polynomial in $e^z$ for which finite order transcendental solutions
are exponential sums and hence of \textit{c.r.g}.
Wen, Gundersen and Heittokangas \cite{DualExp} studied the dual structure
of exponential polynomial coefficients and solutions of $n$-th order equations. 
Li, Wang and Wen \cite{LiWangWen} considered exponential sum coefficients
with rational frequencies and constant multipliers. They obtained
representations of finite order entire solutions as exponential sums,
and existence criteria when the nonzero frequencies have the same sign.
In this case the \textit{c.r.g.}\ conclusion follows from the representation.

   Bergweiler \cite[Theorems~1.1 and~1.2]{Bergweiler}  constructed a second-order equation with coefficients of \textit{c.r.g.}\ having a finite-order entire solution which is not of \textit{c.r.g.}\ This raises the question of when \textit{c.r.g.}\ of the coefficients is inherited by finite-order solutions. In particular, Heittokangas et al.\ \cite[p.~33, Problem~5]{HITW2022} asked whether this holds when all coefficients are exponential polynomials:

    \bigskip
    \noindent
     \emph{
     	If the coefficients $a_0(z),\ldots, a_{n-1}(z)$ of \begin{equation}\label{eq.basic}
     	f^{(n)}+a_{n-1}f^{(n-1)}+\cdots+a_1f'+a_0f=0,
     	\end{equation}
     	 are exponential polynomials, and if it  possesses a transcendental entire solution $f$ of finite order of growth,
     	then is it true that $f$ is of \textit{c.r.g.}?}

    \bigskip
Using Wiman--Valiron theory, Droletz \cite{Droletz} proved that the
possible finite orders of transcendental entire solutions to
\eqref{eq.basic} form a finite set of positive rational numbers,
and that each such solution has order equal to its lower order.

Our proof is based on asymptotic estimates along rays. We combine Wasow's sectorial asymptotic theory for rational systems with an integral-equation argument controlling exponentially small perturbations on almost every ray, in the spirit of \cite{CGM}. This yields an affirmative answer to Heittokangas et al.'s Problem 5.
The ray estimates also give an alternative derivation of the order
conclusions proved by Droletz.

\begin{theorem}\label{Thm.intro}
Suppose that the coefficients $a_0,\ldots,a_{n-1}$ of \eqref{eq.basic}
are exponential polynomials. Every transcendental entire solution of
finite order has positive rational order, finite nonzero type, and
completely regular growth.
\end{theorem}

\bigskip
\noindent\textit{Outline of the proof of Theorem \ref{Thm.intro}.}
Let $f$ be a transcendental entire solution of finite order.
Absorb the constant terms in the exponents into their polynomial
multipliers, and combine terms with the same exponent. The equation becomes
\[
 \sum_{\nu=1}^{N}e^{Q_\nu(z)}
       \sum_{k=0}^{n}a_{k,\nu}(z)f^{(k)}(z)=0,
\]
where the $Q_\nu$ are distinct polynomials with $Q_\nu(0)=0$ and the
$a_{k,\nu}$ are polynomials.
Outside finitely many directions, one exponential factor dominates all the others exponentially.
Gundersen's estimate gives polynomial bounds for $f^{(k)}/f$ on
almost every complete ray tail.
Dividing by the largest exponential factor and by the nonzero polynomial
multiplying the highest derivative in its group gives
\[
 f^{(m)}+P_{m-1}f^{(m-1)}+\cdots+P_0f
 =\epsilon_\theta(r)f,\qquad z=re^{i\theta},
\]
where the $P_k$ are rational and
$|\epsilon_\theta(r)|\le Cr^K e^{-cr}$ with $c>0$.
The case $m=0$ would give $\epsilon_\theta=1$, so $m\ge1$.
Constants and starting radii here may depend on the fixed angle $\theta$.

The equation with zero right-hand side has rational coefficients.
Wasow's theorem gives a fundamental matrix for its first-order system
on a narrow sector containing the ray, outside finitely many further
directions. It has the form
$\mathsf Y_0=\mathsf T e^{\mathsf Q}$, where $\mathsf Q$ is diagonal
with entries that are polynomials in a fractional power of $z$, and
$\mathsf T$ and $\mathsf T^{-1}$ have polynomial bounds.
This transformation reduces the system to a diagonal system with an integrable perturbation on the fixed ray.
Lemma \ref{lem.stein.modify} compares its solutions with the diagonal
system. Together with the derivative bounds, it gives
\[
 \log|f(re^{i\theta})|=\Re q_\theta(re^{i\theta})+O(\log r).
\]
The functions $q_\theta$ come from a finite list determined by the
polynomial coefficient equations.

Finite order and the logarithmic derivative estimate exclude any
selected expression of degree greater than the order of $f$.
A Phragm\'en--Lindel\"of argument then gives positive rational order
and finite nonzero type.
Levin's theorem yields \textit{c.r.g.}\ in the whole plane
from the limits on the remaining dense set of rays.

\bigskip
Section 2 collects the growth and asymptotic tools. Section 3 treats exponential sums and general exponential-polynomial coefficients, and derives consequences for indicators and zeros. Section \ref{Sec.proof} contains the
proofs. Section \ref{sec.concluding} considers entire coefficients beyond the
exponential polynomial class. Finitely many rays divide a neighbourhood
of infinity into sectors $S_\nu$. On each sector, the coefficients,
including $a_n\equiv1$, are assumed to have Laurent--Puiseux asymptotic
expansions of the form
\[
 \frac{a_k(z)}{W_\nu(z)}
 =\sum_{j=j_\nu}^{J}c_{k,\nu,j}z^{-j/p_\nu}
   +O\bigl(|z|^{-(J+1)/p_\nu}\bigr),
 \qquad 0\le k\le n,
\]
for every integer $J\ge j_\nu$, uniformly as $z\to\infty$ on every
closed subsector of $S_\nu$.
Here $W_\nu$ is nonvanishing and holomorphic on $S_\nu$,
$p_\nu\ge1$ and $j_\nu$ are integers, and a branch of $z^{1/p_\nu}$
is fixed. We require
$(c_{k,\nu,j})_{0\le k\le n,\,j\ge j_\nu}\ne0$ for each $\nu$.
Under this condition, Proposition \ref{prop.puiseux} proves that every
transcendental entire solution of finite order has \textit{c.r.g}.
Corollaries \ref{cor.analytic.weights} and
\ref{cor.polynomial.kernels} apply this criterion to integral
coefficients whose normalized expansions involve integer powers
and, respectively, possibly rational fractional powers.

\section{Basic concepts and lemmas}
\subsection{Completely regular growth on rays}\label{Sec.crg}
A set covered by discs satisfying \eqref{r.eq} is called a $C_0$-set.
The relative measure of a measurable set $E\subset(0,\infty)$ is
\[
 m^*(E)=\lim_{r\to\infty}\frac{\operatorname{mes}(E\cap(0,r))}{r},
\]
when this limit exists. In particular, a set of finite linear measure has zero relative measure.

We shall use the constant proximate order $\rho(r)\equiv\rho$.
For this purpose we first need to know that the function has order $\rho>0$
and finite nonzero type, that is,
\[
 0<\limsup_{r\to\infty}\frac{\log M(r,f)}{r^\rho}<\infty,
 \qquad M(r,f)=\max_{|z|=r}|f(z)|.
\]
With respect to this same proximate order, $f$ is of \textit{c.r.g.}\
on the ray $\arg z=\theta$ if
\[
 \lim_{r\to\infty,\ r\notin E_\theta}
 \frac{\log|f(re^{i\theta})|}{r^\rho}
\]
exists and is finite, for some set $E_\theta$ of zero relative measure.
For a set of rays one requires a common exceptional set and uniform convergence.

\begin{lemma}[Levin]\label{lem.levin.lim.crg}
Let $f$ be an entire function of order $\rho>0$ and finite nonzero type.
If $f$ is of \textit{c.r.g.}\ on rays whose arguments form an
everywhere dense set, with respect to the same proximate order $\rho(r)\equiv\rho$,
then $f$ is of \textit{c.r.g.}\ in the whole plane.
\end{lemma}
This is the corollary to the theorem on sets of rays in
\cite[Chapter III, Section 1, pp.~140--142]{Levin1}.
The type assumption permits the use of the constant proximate order;
see also \cite[p.~32]{Levin1}.

\subsection{Asymptotic comparison on rays}\label{Sec.asym}
Considering the differential equation
\begin{equation}\label{lde}
 f^{(n)}+a_{n-1}f^{(n-1)}+\cdots+a_1f'+a_0f=0
\end{equation}
with polynomial coefficients, Petrenko \cite[pp.~104--112]{Petru}
has shown that its transcendental entire solutions are of \textit{c.r.g}.
Steinmetz \cite[Theorem 1]{Stein} proved a more general result for the equation
\begin{equation}\label{lde.p}
 f^{(n)}+R_{n-1}f^{(n-1)}+\cdots+R_1f'+R_0f=0,
\end{equation}
where $R_0,\ldots,R_{n-1}$ are rational functions.

\begin{lemma}[Steinmetz]\label{Thm1.stein}
Let $f$ be a non-rational meromorphic solution of \eqref{lde.p}, and let
$\arg z=\theta$ be an arbitrary direction. There are $p\in\N$, $h>0$,
and a polynomial $Q$ such that
\begin{equation}\label{eq.log|f|}
 \log|f(z)|=\Re Q(z^{1/p})+O(\log|z|)
\end{equation}
as $z\to\infty$ in $\theta\le\arg z\le\theta+h$, possibly outside
\begin{enumerate}
\item discs $|z-z_\mu|<|z_\mu|^{1-\varepsilon}$, where $\varepsilon>0$
and the counting function of the sequence $\{z_\mu\}$ is $O(\log r)$;
\item a logarithmic semi-strip
$0\le\arg z-\theta<\log^+|z|/|z|^{1/p}$.
\end{enumerate}
The latter occurs only for finitely many directions.
\end{lemma}
These exceptional directions are called Stokes rays; see \cite{Stein}
and \cite[Section 15]{Wasow}.

Vector norms are maximum norms and matrix norms are the induced norms. A branch of each fractional power
is fixed on the sector under consideration.

In asymptotic estimates on a fixed ray $\arg z=\theta$, the constants
in the $O$-terms and the starting radius may depend on $\theta$.
These estimates are not asserted to be uniform in $\theta$ unless
this is stated explicitly.

The next lemma follows from Wasow's theorem
\cite[Theorem 19.1, pp.~111--112]{Wasow} and the construction in its proof.

\begin{lemma}\label{thm.Wasow}
Let $\mathsf{A}(z)$ be a rational $m$-by-$m$ matrix.
Near any fixed direction, and for sufficiently large $|z|$, the system
\[
 \mathsf{Y}'=\mathsf{A}(z)\mathsf{Y}
\]
has a fundamental matrix on a sufficiently narrow sector of the form
\begin{equation}\label{matrix.sol}
 \mathsf{Y}_0(z)=\mathsf{T}(z)e^{\mathsf{Q}(z)}.
\end{equation}
Here
\[
 \mathsf Q(z)=\operatorname{diag}
       \bigl(G_1(z^{1/p}),\ldots,G_m(z^{1/p})\bigr),
\]
where $p\in\N$ and the $G_j$ are polynomials with $G_j(0)=0$.
A branch of $z^{1/p}$ is fixed on the sector.
The matrix $\mathsf T$ is holomorphic and invertible.
On every closed subsector there are constants $C,K,R>0$ such that
\begin{equation}\label{eq.T.bounds}
 \|\mathsf{T}(z)\|+\|\mathsf{T}(z)^{-1}\|\le C|z|^K,
 \qquad |z|\ge R.
\end{equation}
Moreover,
\begin{equation}\label{eq.T.identity}
 \mathsf{T}^{-1}\mathsf{A}\mathsf{T}-\mathsf{T}^{-1}\mathsf{T}'=\mathsf{Q}'.
\end{equation}
The integer $p$ and the polynomials $G_1,\ldots,G_m$ can be fixed
independently of the sector, with the diagonal entries reordered if needed.
If a branch of $z^{1/p}$ has been fixed, the other choices are
$e^{2\pi i\ell/p}z^{1/p}$, where $\ell=0,\ldots,p-1$.
Thus the possible diagonal entries belong to the finite family
\[
 G_j\bigl(e^{2\pi i\ell/p}z^{1/p}\bigr),
 \qquad 1\le j\le m,\quad 0\le\ell<p.
\]
\end{lemma}
\begin{proof}
Write $\mathsf{A}(z)=z^d\mathsf{B}(z)$, where $d\ge1$ is an integer and
$\mathsf{B}$ is holomorphic at infinity. By \cite[Theorem 19.1, pp.~111--112]{Wasow}
and its construction in Section 19, an actual fundamental matrix is
\[
 \mathsf{Y}_0(z)=\mathsf{H}(z)z^{\mathsf{G}}e^{\mathsf{Q}(z)},
\]
where $\mathsf{G}$ is constant and $\mathsf{H}$ has an asymptotic Laurent series
in $z^{-1/p}$, with finitely many positive powers allowed.
The matrix of these formal Laurent series is invertible. The construction uses
matrices with invertible leading terms and diagonal matrices of
rational powers of $z$, followed by the regular singular reduction;
see \cite[pp.~58--60, 95--98, 101--111]{Wasow}.
Their products are invertible over the field of formal Laurent series
in $z^{-1/p}$. Hence, on a closed subsector,
\[
 \det\mathsf{H}(z)=c z^{\ell/p}(1+o(1)),\qquad c\ne0,
\]
for an integer $\ell$. The entries of $\mathsf{H}$ are bounded by powers
of $|z|$, and the adjugate formula gives the same bound for $\mathsf{H}^{-1}$.
The matrices $z^{\mathsf{G}}$ and $z^{-\mathsf{G}}$ also have polynomial bounds
on a bounded argument interval: their Jordan blocks contain only
powers of $z$ multiplied by polynomials in $\log z$.
Thus $\mathsf{T}=\mathsf{H} z^{\mathsf{G}}$ satisfies \eqref{eq.T.bounds}.
Constant terms in the diagonal polynomials can be absorbed into $\mathsf{T}$,
so these polynomials may be taken to have zero constant term.
In Wasow's construction, the integer $p$ and the polynomials $G_j$
are obtained from the rational system before the sectorial fundamental
matrices are constructed. They can therefore be kept fixed when the
sector changes; the different choices of $z^{1/p}$ give the finite
family displayed above.
Differentiating \eqref{matrix.sol} gives
\[
 \mathsf{T}'+\mathsf{T}\mathsf{Q}'=\mathsf{A}\mathsf{T},
\]
which proves \eqref{eq.T.identity}.
\end{proof}

\begin{lemma}\label{lem.stein.modify}
Let $q_1,\ldots,q_m$ be complex-valued $C^1$ functions on a half-line.
Suppose each difference $q_i-q_j$ is either identically zero or has
real part eventually strictly monotone and tending to $+\infty$ or $-\infty$.
If $\mathsf{R}$ is continuous and
$\int_{r_0}^\infty\|\mathsf{R}(r)\|\,dr<\infty$, then the system
\[
 \mathsf{Z}'=(\operatorname{diag}(q_1',\ldots,q_m')+\mathsf{R})\mathsf{Z}
\]
has a fundamental matrix
\begin{equation}\label{eq.ray.fundamental}
 \mathsf{Z}(r)=(\mathsf{I}+o(1))\operatorname{diag}(e^{q_1(r)},\ldots,e^{q_m(r)}).
\end{equation}
\end{lemma}
\begin{proof}
Fix a column $j$, and write it as $e^{q_j}\bm u$.
Put $\Delta_i=q_i-q_j$. Choose $r_0$ large enough that all the relevant
real differences are monotone and
\[
 a:=\int_{r_0}^\infty\|\mathsf{R}(s)\|\,ds<\frac12.
\]
On the Banach space of bounded continuous vectors on $[r_0,\infty)$,
with the supremum norm, consider
\begin{equation}\label{eq.ray.integral}
 u_i(r)=
 \begin{cases}
 \displaystyle\int_{r_0}^r e^{\Delta_i(r)-\Delta_i(s)}
                (\mathsf{R}(s)\bm u(s))_i\,ds,
       &\Re\Delta_i\to-\infty,\\[6pt]
 \displaystyle-\int_r^\infty e^{\Delta_i(r)-\Delta_i(s)}
                (\mathsf{R}(s)\bm u(s))_i\,ds,
       &\Re\Delta_i\to+\infty,\\[6pt]
 \displaystyle\delta_{ij}-\int_r^\infty(\mathsf{R}(s)\bm u(s))_i\,ds,
       &\Delta_i\equiv0.
 \end{cases}
\end{equation}
Every exponential kernel has modulus at most $1$ on its integration interval.
The integrals define bounded continuous functions, and the norm of the
linear part of this map is at most $a$. It is therefore a contraction
and has a unique fixed point.

The integrals to infinity tend to zero, being bounded by
$\|\bm u\|_\infty\int_r^\infty\|\mathsf{R}(s)\|\,ds$.
For a forward integral, fix $S>r_0$ and split it into $[r_0,S]$ and $[S,r]$.
On the first interval the kernel tends uniformly to zero as $r\to\infty$,
because $\Re\Delta_i(r)\to-\infty$.
The second integral is bounded by
$\|\bm u\|_\infty\int_S^\infty\|\mathsf{R}(s)\|\,ds$.
First let $r\to\infty$, and then let $S\to\infty$.
It follows that $\bm u(r)\to\bm e_j$.

Differentiation of the forward integrals is immediate.
For a backward integral, write it as
$-e^{\Delta_i(r)}\int_r^\infty
e^{-\Delta_i(s)}(\mathsf{R}(s)\bm u(s))_i\,ds$.
The integral converges absolutely, so the fundamental theorem of calculus
gives the same derivative. Thus each fixed point satisfies
\[
 u_i'=\Delta_i'u_i+(\mathsf{R}\bm u)_i.
\]
The columns $e^{q_j}\bm u$ solve the differential system.
Their normalized matrix tends to $\mathsf{I}$, so it is invertible for large $r$.
This proves \eqref{eq.ray.fundamental}.
\end{proof}

\begin{lemma}\label{lem.ray.comparison}
Let
\[
 f^{(m)}+P_{m-1}(z)f^{(m-1)}+\cdots+P_0(z)f
=\epsilon_\theta(r)f,\qquad z=re^{i\theta},
\]
hold on a ray tail, where $m\ge1$, the $P_k$ are rational, and
$\epsilon_\theta$ is continuous with
$|\epsilon_\theta(r)|\le Cr^K e^{-cr}$ for some $c>0$.
Suppose $f$ is holomorphic near the ray, is eventually nonzero there, and
\[
 |f^{(k)}(re^{i\theta})/f(re^{i\theta})|\le Cr^K,
 \qquad 1\le k<m.
\]
There is a finite set of directions, determined by the rational equation,
outside which
\begin{equation}\label{eq.log|f|.sector}
 \log|f(re^{i\theta})|=\Re G_\theta((re^{i\theta})^{1/p})
                         +O(\log r).
\end{equation}
The polynomial $G_\theta$ belongs to a finite list, with the branches fixed
on finitely many argument intervals. A common denominator $p$ may be
chosen for this finite list.
\end{lemma}
\begin{proof}
Put $\bm x=(f,f',\ldots,f^{(m-1)})^{\mathsf T}$, and let $\mathsf{A}$
be the companion matrix of the rational equation with zero right-hand side.
Its last row is $(-P_0,\ldots,-P_{m-1})$ and its superdiagonal entries are $1$.
Let $\mathsf{E}=\bm e_m\bm e_1^{\mathsf T}$. Solving the scalar equation for
$f^{(m)}$ and using $dz/dr=e^{i\theta}$ gives the exact system
\begin{equation}\label{eq.ray.companion}
 \frac{d\bm x}{dr}=e^{i\theta}(\mathsf{A}(re^{i\theta})
                                 +\epsilon_\theta(r)\mathsf{E})\bm x.
\end{equation}
Apply Lemma \ref{thm.Wasow} to the unperturbed rational system
$\mathsf{Y}'=\mathsf{A}\mathsf{Y}$ on a whole narrow sector containing this ray.
Set $\bm x(re^{i\theta})=\mathsf{T}(re^{i\theta})\bm z(r)$.
The chain rule and \eqref{eq.T.identity} give
\begin{equation}\label{eq.ray.gauge}
 \frac{d\bm z}{dr}=
 \left[\frac d{dr}\mathsf{Q}(re^{i\theta})+\mathsf{R}_\theta(r)\right]\bm z,
 \qquad
 \mathsf{R}_\theta=e^{i\theta}\mathsf{T}^{-1}\epsilon_\theta\mathsf{E}\mathsf{T}.
\end{equation}
By \eqref{eq.T.bounds}, $\mathsf{R}_\theta$ is integrable.

Write the diagonal entries of $\mathsf{Q}$ as $q_1(z),\ldots,q_m(z)$.
Exclude the directions where the leading term of a nonzero $q_i$
or a nonzero $q_i-q_j$ has zero real part.
There are only finitely many such directions for all choices of the
functions $q_i$ given by Lemma \ref{thm.Wasow}.
On each remaining ray, $\Re(q_i-q_j)$ is either identically zero
with $q_i=q_j$, or is eventually strictly monotone and tends to
$+\infty$ or $-\infty$.
Lemma \ref{lem.stein.modify} gives, for the particular solution $\bm x$,
\[
 \bm x(re^{i\theta})=
 \mathsf{T}(re^{i\theta})(\mathsf{I}+o(1))
 \operatorname{diag}(e^{q_1(re^{i\theta})},\ldots,e^{q_m(re^{i\theta})})\bm c,
\]
where $\bm c\ne0$.
The polynomial bounds for $\mathsf{T}$ and its inverse imply
\[
 \log\|\bm x(re^{i\theta})\|
 =\max_{c_i\ne0}\Re q_i(re^{i\theta})+O(\log r).
\]
The hypotheses on the derivatives give
$|f|\le\|\bm x\|\le Cr^K|f|$, so the same formula holds for $\log|f|$.
The real parts $\Re q_i(re^{i\theta})$ have a fixed ordering for large $r$.
Thus the maximum is one fixed $q_i=G_\theta(z^{1/p})$ on this ray tail,
which proves \eqref{eq.log|f|.sector}.
\end{proof}

\begin{lemma}\label{lem.order.crg}
Let $f$ be a transcendental entire function of finite order $\rho$.
Suppose there is a dense set $\mathbb M$ of directions and a finite list
of polynomials in fractional powers of $z$, on fixed argument intervals,
such that for every $\theta\in\mathbb M$,
\begin{equation}\label{eq.phase.hypothesis}
 \log|f(re^{i\theta})|=\Re q_\theta(re^{i\theta})+O(\log r),
\end{equation}
with $q_\theta$ from this list. Suppose also that for each $\delta>0$,
\begin{equation}\label{eq.two.sided.hypothesis}
 |\log|f(re^{i\theta})||\le C_{\theta,\delta}r^{\rho+\delta}
\end{equation}
on every such ray tail. If the highest term of a nonzero selected
$q_\theta$ is $a z^d$, suppose also that $\Re(a e^{id\theta})\ne0$.
Then $\rho$ is positive and rational, $f$ has finite nonzero type, and
$f$ is of \textit{c.r.g}.
\end{lemma}
\begin{proof}
For $q(z)=G(z^{1/p})\ne0$, define its degree to be $\deg G/p$;
the zero function is assigned degree $0$.
If a selected $q_\theta$ had degree $d>\rho$, then \eqref{eq.phase.hypothesis}
would contradict \eqref{eq.two.sided.hypothesis}, by choosing
$0<\delta<d-\rho$.
Thus every selected degree is at most $\rho$.
Let $\sigma$ be the largest degree actually selected on $\mathbb M$.
This maximum exists because the list is finite, and $\sigma\le\rho$.

If $0<\tau\le\rho$ and every ray in a dense set satisfies
$|f(re^{i\theta})|\le\exp(C_\theta r^\tau)$ eventually, then
\begin{equation}\label{eq.PL.bound}
 \log M(r,f)=O(r^\tau).
\end{equation}
To see this, choose finitely many of these rays around the circle, with
each consecutive angular gap $\beta<\pi/(\rho+1)$.
Choose one inner radius beyond the finitely many starting radii.
In each exterior sector between adjacent rays fix a branch of $z^\tau$,
and let $\theta_0$ be its middle argument.
Then
\[
 \Re(e^{-i\tau\theta_0}z^\tau)
 =r^\tau\cos(\tau(\arg z-\theta_0))\ge c_0r^\tau
\]
on the closed sector, with $c_0>0$.
For sufficiently large $K$, the function
\[
 f(z)\exp(-K e^{-i\tau\theta_0}z^\tau)
\]
is bounded on both boundary rays
and on the fixed inner circular arc. Its growth order is at most $\rho$,
which is smaller than $\pi/\beta$.
The sectorial Phragm\'en--Lindel\"of theorem makes it bounded in the sector.
The finitely many sectors give \eqref{eq.PL.bound}.
If instead the bounds on each ray are polynomial, choose this same finite
grid and one integer $N$ larger than their exponents.
Apply the argument to $f(z)/z^N$ on the exterior sectors.
Then $M(r,f)=O(r^N)$, and Cauchy's estimates show that $f$ is a polynomial.

If $\sigma=0$, \eqref{eq.phase.hypothesis} gives polynomial bounds on all
the rays in $\mathbb M$, a contradiction. Hence $\sigma>0$.
The same formula gives $|f(re^{i\theta})|\le\exp(C_\theta r^\sigma)$.
By \eqref{eq.PL.bound}, $\rho\le\sigma$ and $f$ has finite type at order $\sigma$.
Consequently $\rho=\sigma$ is positive and rational.

This type is nonzero. Otherwise no selected $q_\theta$ of degree $\rho$ could
have a positive real leading coefficient, since one such ray already gives
$\log M(r,f)\ge c r^\rho$ for large $r$.
All selected $q_\theta$ of degree $\rho$ would therefore have leading terms
with negative real parts on their rays, and $f$ would be bounded there.
All other selected degrees belong to a finite list below $\rho$.
If its largest degree is positive, \eqref{eq.PL.bound} would give an order
strictly below $\rho$; if there is no positive degree, the polynomial version
would make $f$ a polynomial. Both conclusions are impossible.

Finally, division of \eqref{eq.phase.hypothesis} by $r^\rho$ gives a finite
limit on every ray in $\mathbb M$. The constant $\rho$ is now a proximate
order of $f$. Lemma \ref{lem.levin.lim.crg} proves \textit{c.r.g}.
\end{proof}

\section{Coefficient equations and ray asymptotics}\label{Sec.result}
	\subsection{Exponential sum coefficients}

	We first consider the linear differential equation 
	 \begin{equation}\label{expdiff.eq}
		f^{(n)}(z)+A_{n-1}(z)f^{(n-1)}(z)+\cdots+A_1(z)f'(z)+A_0(z)f(z)=0
	\end{equation}
	with exponential sum coefficients $A_i(z)$ for $i=0,1,\ldots,n-1$ of the form \eqref{expsum.eq}. The equation \eqref{expdiff.eq} can be transformed into a normal form
	\begin{equation}\label{standard.eq}
		\sum_{j=0}^{N_0}e^{\lambda_j z}\left(a_{n,j}(z)f^{(n)}(z)+a_{n-1,j}(z)f^{(n-1)}(z)+\cdots +a_{0,j}(z)f(z)\right)=0,\quad 
	\end{equation}
The notation used below is as follows.
\begin{enumerate}
\item \textit{Frequency groups.}
The $a_{k,j}$ are polynomials, and the distinct frequencies $\lambda_j$
include $0$. The term $f^{(n)}$ belongs to this zero-frequency group:
$a_{n,j}=1$ when $\lambda_j=0$, and $a_{n,j}=0$ otherwise.
Groups in which all coefficients vanish are omitted.

\item \textit{Coefficient equations.}
For each frequency $\lambda_j$, we call
\[
 a_{n,j}(z)f^{(n)}+a_{n-1,j}(z)f^{(n-1)}+\cdots+a_{0,j}(z)f=0
\]
the coefficient differential equation at $\lambda_j$.
The given solution need not satisfy these equations separately.

\item \textit{The convex hull.}
Put $W=\{\overline\lambda_j:0\le j\le N_0\}$.
Write the vertices of its convex hull as
\[
 \tilde W=\{\overline\omega_j:0\le j\le N_0'\},
\]
in anticlockwise order. There are $N_0'+1$ vertices, with indices read
cyclically. Number the frequency groups so that $\lambda_j=\omega_j$
for $0\le j\le N_0'$.
We use conjugated frequencies because $\Re(\lambda z)$ is the Euclidean
scalar product of $\overline\lambda$ and $z$.

\item \textit{Critical rays.}
Let $s_j$ join $\overline\omega_j$ to $\overline\omega_{j+1}$.
The outward normal to this side has argument $\eta_j$.
The ray $\arg z=\eta_j$ is called a critical ray.

\item \textit{Vertex sectors and boundary regions.}
For sufficiently small $\varepsilon>0$, put
\[
 \begin{aligned}
 \overline S_j
   &=\{z\ne0:\eta_{j-1}\le\arg z\le\eta_j\},\\
 S_j(\varepsilon)
   &=\{z\ne0:\eta_{j-1}+\varepsilon\le\arg z\le\eta_j-\varepsilon\},\\
 T_j(\varepsilon)
   &=\{z\ne0:|\arg z-\eta_j|<\varepsilon\}.
 \end{aligned}
\]
Arguments are chosen continuously across the sector in question.
The frequency $\omega_j$ is dominant in the interior of $\overline S_j$.
The set $S_j(\varepsilon)$ stays away from its boundary rays, while
$T_j(\varepsilon)$ is a small angular neighbourhood of a critical ray.
\end{enumerate}
Figure \ref{fig.convex.hull.1} shows this construction.
If the hull is a segment, its two normal rays bound the two vertex
half-planes. If it is a single point, there are no critical rays and
the coefficient equation is the whole equation.

\begin{figure}[h]
  	\begin{tikzpicture}
  		\draw[<->](5.5,0)--(0,0)--(0,5.3);
  		\draw(-5,0)--(0,0)--(0,-3);
  		\draw[ -](0.5,1)--(1.5,4)--(4,4.5)--(5,1)--(3.5,-1)--(0.5,1)  node at (0.5,0.6){$\overline{\omega}_{j+1}$} node at (1.5,4.3){$\overline{\omega}_{j}$}
  		node at (4,4.8){$\overline{\omega}_{j-1}$}
  		node at (5.4,1.0){$\overline{\omega}_{0}$}
  		node at
  		(3.5,-1.3){$\overline{\omega}_{N_0'}$}
  		node at
  		(3,2.2){${\tilde{W}}$};
  		\draw[domain=-5:0] plot(\x,-0.25*\x) node at (-5.3,1.5){$\arg z=\eta_j$};
  		\draw[dashed,domain=-4.8:0] plot(\x,-0.4*\x);
  		\draw[dashed,domain=-5:0] plot(\x,-0.08*\x);
  		\draw[<->]	(-2,0.5)	arc(170:155:1)
  		node at (-2.3,0.75){$\varepsilon$};
  		\draw[<->]	(-2,0.5)	arc(170:175:4)
  		node at (-2.4,0.35){$\varepsilon$};
  		\draw[dashed,domain=-1.6:0] plot(\x,-3*\x);
  		\draw [<->]	(-0.4,2)	arc(110:117:2)
  		node at (-0.6,2.4){$\varepsilon$};
  		\draw [<->]	(-1,3)	arc(110:155:3.5)
  		node at (-2.3,2.8){$S_j(\varepsilon)$};
  		\draw [<->]	(-0.9,4.5)	arc(110:161.5:5.5)
  		node at (-3,4){$\overline{S}_j$};
  		\draw [<->]	(-3.5,1.4)	arc(150:172:3)
  		node at (-4.3,0.7){$T_j(\varepsilon)$};
  		\draw[domain=-1:0,smooth] plot(\x,-5*\x) node at (-1,5.6){$\arg z=\eta_{j-1}$};
  		\draw[domain=0:1.5] plot(3.5*\x,\x) node at (5.5,1.5){$\eta_0$};
  		\draw[domain=-0.8:0] plot(2*\x,3*\x) node at (-1.6,-2.6){$\eta_{j+1}$};
  		\draw[domain=0:1.6] plot(2*\x,-1.5*\x) node at (3.6,-2.6){$\eta_{N_0'}$};
  	\end{tikzpicture}
  \caption{The conjugate-frequency hull, its critical rays, and a vertex sector.}
  \label{fig.convex.hull.1}
  \end{figure}

	\begin{lemma}\label{lem.convex.Sj}
If $\overline\lambda\in W$, then for every $z\in S_j(\varepsilon)$,
\[
 \Re((\omega_j-\lambda)z)
 \ge |z|\,|\omega_j-\lambda|\,\sin\varepsilon.
\]
In particular, for a fixed interior ray of a vertex sector and
$\lambda\ne\omega_j$, the right-hand side is at least $c_\theta|z|$
for some $c_\theta>0$.
\end{lemma}
See also Dickson \cite{Dickson}.

Our conclusions are as follows.

\begin{theorem}\label{Thm.main}
Every finite order transcendental entire solution $f$ of \eqref{expdiff.eq}
has positive rational order, finite nonzero type, and \textit{c.r.g}.
There is a set $\mathcal E_f\subset[0,2\pi)$ of measure zero such that,
for every $\theta\notin\mathcal E_f$ in a vertex sector at $\overline\omega_j$,
\[
 \log|f(re^{i\theta})|
 =\Re G_{j\theta}((re^{i\theta})^{1/p_j})+O(\log r),
 \qquad r\to\infty,
\]
on the complete ray tail. Here $p_j\in\N$ and $G_{j\theta}$ is a polynomial
from a finite list determined by the coefficient differential equation
at the frequency $\omega_j$; its branch is fixed on the argument interval.
The polynomial, the error constant and the starting radius may depend on $\theta$.
\end{theorem}
\begin{remark}
The error estimate above is on each fixed ray. We do not claim this
$O(\log r)$ error uniformly in a sector. Completely regular growth
still gives the uniform $o(r^{\rho(f)})$ estimate \eqref{crg} outside a $C_0$-set.
\end{remark}

\begin{example}\label{ex.exp-sum.1}
The function $f(z)=e^{-4z/3}(1-7e^z)$ from
\cite[Example~3.6]{LiWangWen} solves
\[
 \left(f'''-\frac43f'+\frac{16}{27}f\right)
       +e^z(3f''-2f'-f)=0.
\]
The two coefficient expressions, divided by $f$, are
$-7e^z/(1-7e^z)$ and $7/(1-7e^z)$.
They are exponentially small in closed subsectors of the left and right
half-planes, respectively. Thus the order of the dominant coefficient
equation changes between sectors: it is three in the left half-plane
and two in the right half-plane.
\end{example}

\subsection{Exponential polynomial coefficients}\label{subsec.e-p.coeff}
We now consider
\begin{equation}\label{lde.exp-poly}
 f^{(n)}+G_{n-1}(z)f^{(n-1)}+\cdots+G_1(z)f'+G_0(z)f=0,
\end{equation}
where
\begin{equation}\label{coef.exp-poly}
 G_m(z)=\sum_{t=1}^{l_m}P_{m,t}(z)e^{Q_{m,t}(z)},
 \qquad 0\le m<n,
\end{equation}
and $P_{m,t},Q_{m,t}$ are polynomials. Zero coefficients are allowed.
We first group the terms with the same exponential factor.
We then compare the coefficients in their exponents, starting with the highest degree.

\begin{enumerate}
\item \textit{Grouping equal exponents.}
Absorb the constant term of each exponent into its polynomial multiplier,
combine equal exponents, and include the exponent $0$ from $f^{(n)}$.
After omitting zero groups, write the equation as
\begin{equation}\label{eq.full.exponents}
 \sum_{\nu=1}^{N}e^{Q_\nu(z)}
       \left\{\sum_{m=0}^{n}a_{m,\nu}(z)f^{(m)}\right\}=0.
\end{equation}
Here $Q_\nu(0)=0$, the $Q_\nu$ are distinct, and the $a_{m,\nu}$ are
polynomials. Each coefficient group is nonzero.

\item \textit{Grouping by degree.}
If all $Q_\nu$ have degree at most one, Theorem \ref{Thm.main} applies.
Otherwise put $s=\max_\nu\deg Q_\nu\ge2$ and write
\[
 Q_\nu(z)=q_{\nu,s}z^s+q_{\nu,s-1}z^{s-1}+\cdots+q_{\nu,1}z.
\]
Thus $q_{\nu,d}$ means only the coefficient of $z^d$ in $Q_\nu$;
missing coefficients are set equal to zero.
First group the terms with the same $q_{\nu,s}$.
Within each group, group the terms with the same $q_{\nu,s-1}$,
and continue down to $q_{\nu,1}$.
At each step, the coefficients of all higher powers are held fixed.
Figure \ref{fig.leading.coeff.} illustrates this finite procedure.
Only one branch is expanded in the figure; each leaf represents
one of the polynomials $Q_\nu$.

\item \textit{Fundamental coefficient equations.}
After the coefficient of $z$ has been fixed, each group contains a single
exponent $Q_\nu$. We call
\begin{equation}\label{coeff.eq}
 \sum_{m=0}^{n}a_{m,\nu}(z)f^{(m)}=0
\end{equation}
its fundamental coefficient differential equation.
These equations have polynomial coefficients.
\end{enumerate}

\begin{figure}[ht]
\centering
\begin{tikzpicture}[x=1cm,y=1cm, every node/.style={font=\small}]
\node (root) at (0,0) {$Q_1,\ldots,Q_N$};
\node[draw,minimum width=1.35cm] (a) at (-2.8,-1.0) {group};
\node[draw,minimum width=1.35cm] (b) at (0,-1.0) {group};
\node (dots) at (1.4,-1.0) {$\cdots$};
\node[draw,minimum width=1.35cm] (c) at (2.8,-1.0) {group};
\draw[->] (root)--(a);
\draw[->] (root)--(b);
\draw[->] (root)--(c);
\node[anchor=east] at (-3.8,-1.0) {same $q_{\nu,s}$};
\node[draw,minimum width=1.45cm] (b1) at (-1.2,-2.1) {subgroup};
\node[draw,minimum width=1.45cm] (b2) at (1.2,-2.1) {subgroup};
\draw[->] (b)--(b1);
\draw[->] (b)--(b2);
\node[anchor=east] at (-3.8,-2.1) {same $q_{\nu,s-1}$};
\node (v1) at (-1.2,-2.9) {$\vdots$};
\node (v2) at (1.2,-2.9) {$\vdots$};
\draw[->] (b1)--(v1);
\draw[->] (b2)--(v2);
\node (leaf1) at (-1.2,-3.7) {$Q_\nu$};
\node (leaf2) at (1.2,-3.7) {$Q_\mu$};
\draw[->] (v1)--(leaf1);
\draw[->] (v2)--(leaf2);
\node[anchor=east] at (-3.8,-3.7) {continue down to $z$};
\end{tikzpicture}
\caption{Successive grouping of the exponents.}
\label{fig.leading.coeff.}
\end{figure}

The angular construction uses the same notation at each step.
\begin{enumerate}
\item \textit{Fix the higher-degree coefficients.}
At the step concerning $z^d$, consider one of the groups already formed
(the full set of exponents at the first step).
Apply the convex-hull construction of Section 3.1 to the conjugates
of the remaining values $q_{\nu,d}$ in this group.
For this fixed step, use the same local notation
$W,\overline\omega_j,\eta_j,S_j(\varepsilon),T_j(\varepsilon)$ as before,
now in the variable $w=z^d$.

\item \textit{Return to the $z$-plane.}
A critical ray $\arg w=\eta_j$ has the inverse images
\[
 \arg z=\frac{\eta_j+2\pi\ell}{d},\qquad \ell=0,\ldots,d-1.
\]
In particular, $S_j(\varepsilon)$ has $d$ inverse-image sectors:
\[
 \frac{\eta_{j-1}+\varepsilon+2\pi\ell}{d}
 \le\arg z\le
 \frac{\eta_j-\varepsilon+2\pi\ell}{d},
 \qquad \ell=0,\ldots,d-1.
\]
Their angular widths are divided by $d$, and consecutive sectors are
rotated by $2\pi/d$.
A single remaining value introduces no further sector restriction.
For a segment hull, use the two half-planes as in Section 3.1.

\item \textit{Continue within the selected sectors.}
At each subsequent step, intersect these inverse-image sectors with
the sector already selected, and retain the nonempty intersections.
On an interior ray, the procedure selects
the same $Q_\nu$ as direct comparison of these polynomials.
The proof below makes this comparison directly.
\end{enumerate}

Figure \ref{fig.convex.hull.exp-poly} illustrates these steps for $d=3$.
Under $w=z^3$, the sector $30^\circ\le\arg w\le120^\circ$ has three
inverse-image sectors, with argument intervals
$[10^\circ,40^\circ]$, $[130^\circ,160^\circ]$, and $[250^\circ,280^\circ]$.
If an earlier step has selected $20^\circ\le\arg z\le60^\circ$
(dashed boundaries in the last panel), only the shaded sector
$20^\circ\le\arg z\le40^\circ$ remains. The arcs indicate angular ranges only.

\begin{figure}[ht]
\centering
\begin{tikzpicture}[x=1cm,y=1cm,every node/.style={font=\small}]
\begin{scope}
\node at (0,2.3) {$w$-plane};
\fill[black!16] (0,0)--(30:1.5) arc (30:120:1.5)--cycle;
\draw[gray!50,->] (-1.65,0)--(1.7,0);
\draw[gray!50,->] (0,-1.65)--(0,1.8);
\draw (0,0)--(30:1.5) arc (30:120:1.5)--cycle;
\node[anchor=west] at (30:1.6) {$30^\circ$};
\node[anchor=south east] at (120:1.6) {$120^\circ$};
\node at (0,-2.05) {one selected sector};
\end{scope}
\begin{scope}[xshift=5cm]
\node at (0,2.3) {$z$-plane, $d=3$};
\foreach \a/\b in {10/40,130/160,250/280}{
  \fill[black!16] (0,0)--(\a:1.5) arc (\a:\b:1.5)--cycle;
}
\draw[gray!50,->] (-1.65,0)--(1.7,0);
\draw[gray!50,->] (0,-1.65)--(0,1.8);
\foreach \a/\b in {10/40,130/160,250/280}{
  \draw (0,0)--(\a:1.5) arc (\a:\b:1.5)--cycle;
}
\node at (26:2.0) {$\ell=0$};
\node at (145:1.95) {$\ell=1$};
\node[anchor=west] at (282:1.55) {$\ell=2$};
\node at (0,-2.05) {three inverse-image sectors};
\end{scope}
\begin{scope}[xshift=10cm]
\node at (0,2.3) {after intersection};
\fill[black!22] (0,0)--(20:1.5) arc (20:40:1.5)--cycle;
\draw[gray!50,->] (-1.65,0)--(1.7,0);
\draw[gray!50,->] (0,-1.65)--(0,1.8);
\draw (0,0)--(40:1.5) arc (40:20:1.5);
\draw[dashed] (0,0)--(20:1.72);
\draw[dashed] (0,0)--(60:1.72);
\node[anchor=west] at (20:1.78) {$20^\circ$};
\node[anchor=south west] at (40:1.53) {$40^\circ$};
\node[anchor=south] at (60:1.78) {$60^\circ$};
\node at (0,-2.05) {retain the overlap};
\end{scope}
\draw[->] (2.05,-0.25)--(2.95,-0.25)
  node[midway,above,font=\scriptsize] {pull back}
  node[midway,below,font=\scriptsize] {$w=z^3$};
\draw[->] (7.05,-0.25)--(7.95,-0.25)
  node[midway,above,font=\scriptsize] {intersect};
\end{tikzpicture}
\caption{Inverse-image sectors and their intersection.}
\label{fig.convex.hull.exp-poly}
\end{figure}

\begin{lemma}\label{lem.full.dominance}
For finitely many polynomials $Q_1,\ldots,Q_N$ with nonconstant pairwise
differences, there is a finite set of directions such that, outside this
set, a unique $Q_\nu$ is dominant for large $r$ on $z=re^{i\theta}$.
For every other $Q_\mu$ there is $c_\theta>0$ such that
\begin{equation}\label{eq.full.dominance}
 \Re(Q_\mu(re^{i\theta})-Q_\nu(re^{i\theta}))\le-c_\theta r
\end{equation}
for all sufficiently large $r$.
The dominant index is constant on every component of the remaining
angular set, and the estimate is uniform on closed subintervals.
\end{lemma}
\begin{proof}
For each difference write its highest nonzero term as $c z^d$, where
$d\ge1$. Exclude the finitely many directions where
$\Re(c e^{id\theta})=0$. On every remaining angular interval the real
parts of all pairwise differences have fixed signs for large $r$.
They therefore give a strict ordering and a unique largest exponent.
On a closed subinterval the real leading coefficients are bounded away
from zero. For a difference with negative real leading term,
its lower terms are $O(r^{d-1})$, so its real part is at most
$-c r^d\le-c r$ for large $r$. Taking the minimum over the finite
set of differences proves the assertion.
\end{proof}

Now a similar conclusion to Theorem \ref{Thm.main} follows.

\begin{theorem}\label{Thm.main.exp-poly}
Every finite order transcendental entire solution $f$ of \eqref{lde.exp-poly}
has positive rational order, finite nonzero type, and \textit{c.r.g}.
There is a set $\mathcal E_f\subset[0,2\pi)$ of measure zero such that,
for every $\theta\notin\mathcal E_f$,
\begin{equation}\label{eq.log|f|.exp-poly.coeffi.}
 \log|f(re^{i\theta})|
 =\Re G_\theta((re^{i\theta})^{1/N_\theta})+O(\log r),
 \qquad r\to\infty,
\end{equation}
on the complete ray tail, where $N_\theta\in\N$ and $G_\theta$ is a polynomial.
The functions $G_\theta(z^{1/N_\theta})$ belong to a finite list determined
by the fundamental coefficient differential equations. The branch is fixed on its argument interval.
The polynomial, the error constant and the starting radius may depend on $\theta$.
\end{theorem}
\begin{remark}
The conclusion of \textit{c.r.g.}\ is the uniform statement
\eqref{crg} outside a $C_0$-set. The error $O(\log r)$ in
\eqref{eq.log|f|.exp-poly.coeffi.} is on fixed rays.
Corollary \ref{cor.indicator.zeros} describes the highest-degree terms
that determine the indicator. It does not describe changes in the
lower-degree terms of $G_\theta(z^{1/N_\theta})$.
\end{remark}

\begin{remark}\label{rem.nonmonic}
The same growth conclusion holds for every finite order transcendental
entire solution of
\begin{equation}\label{eq.nonmonic}
 A_n f^{(n)}+A_{n-1}f^{(n-1)}+\cdots+A_0f=0,
\end{equation}
where $A_0,\ldots,A_n$ are exponential polynomials and
$A_n\not\equiv0$.
Indeed, group the terms as in \eqref{eq.full.exponents} and, on each
remaining ray, divide by the largest exponential factor and by the
polynomial multiplying the highest derivative in its group.
The latter polynomial has only finitely many zeros, and the other
groups give the same exponentially small right-hand side as in
\eqref{eq.exppoly.small}. A group of order zero would give
$1=O(r^K e^{-cr})$, which is impossible.
Lemmas \ref{lem.ray.comparison} and \ref{lem.order.crg} therefore
apply without change.
Droletz's results on the possible orders already include these equations
\cite{Droletz}.
\end{remark}

\subsection{Indicators and zeros}\label{subsec.consequences}
The rays in the following statement are directions where the formula
for the indicator may change.

\begin{corollary}\label{cor.indicator.zeros}
Fix an equation \eqref{eq.nonmonic}, where $A_0,\ldots,A_n$ are
exponential polynomials and $A_n\not\equiv0$.
\begin{enumerate}
\item Only finitely many pairs $(\rho(f),h_f)$ occur among its
transcendental entire solutions of finite order.
There are finitely many rays $\arg z=\vartheta_j$, $1\le j\le s$,
depending only on the equation, such that each such solution of
order $\rho$ has
\[
 h_f(\theta)=\Re(c e^{i\rho\theta})
\]
between consecutive rays. The constant $c$ depends on the solution
and the interval; it may be zero.

\item Outside any fixed angular neighbourhoods of these rays, the
number of zeros in $|z|\le r$, counted with multiplicities, is
$o(r^\rho)$.
In particular, this estimate holds in every closed angular sector
separated from these rays.
\end{enumerate}
\end{corollary}

The angular densities are given more precisely as follows.
Let $n_f(r;\alpha,\beta)$ count, with multiplicities, the zeros in
$0<|z|\le r$, $\alpha<\arg z<\beta$.
Represent the angles by $0\le\vartheta_1<\cdots<\vartheta_s<2\pi$.
If $0\le\alpha<\beta\le2\pi$ and neither boundary ray is among the
listed rays, then
\begin{equation}\label{eq.zero.angular.density}
 n_f(r;\alpha,\beta)
 =\frac{r^\rho}{2\pi\rho}
   \sum_{\alpha<\vartheta_j<\beta}
   \bigl(h_f'(\vartheta_j+0)-h_f'(\vartheta_j-0)\bigr)
   +o(r^\rho).
\end{equation}
Every difference in the sum is nonnegative.
For a particular solution, some of these differences may be zero.
The estimate in the corollary concerns angular density; it does not
assert that all zeros lie on the listed rays.

\section{Proofs}\label{Sec.proof}
\subsection{Proof of Theorem \ref{Thm.main}}
\begin{proof}
Let $f$ be a transcendental entire solution of finite order $\rho$.
By Gundersen \cite[Corollary 1, p.~89]{Gundersen.deriv.}, applied to
$\{(k,0):1\le k\le n\}$, for each $\delta>0$ there is an angular
exceptional set of measure zero such that
\begin{equation}\label{eq.ray.Gundersen}
 \left|\frac{f^{(k)}(re^{i\theta})}{f(re^{i\theta})}\right|
 \le r^{k(\rho-1+\delta)},\qquad 1\le k\le n,
\end{equation}
on every remaining ray, for all sufficiently large $r$.
Apply this with $\delta=1/\ell$, $\ell\in\N$, and take the countable
union of the angular exceptional sets. We obtain a single null set
$\mathcal E_G$ for which \eqref{eq.ray.Gundersen} holds for every
$\delta>0$. The starting radius may depend on $\theta$ and $\delta$.

Each such ray is eventually zero-free. Indeed, if a zero of multiplicity
$s$ lay on its tail at $z_0=r_0e^{i\theta}$, then
$f'(z)/f(z)=s/(z-z_0)+O(1)$ near $z_0$.
Approaching $r_0$ on this ray contradicts \eqref{eq.ray.Gundersen} for $k=1$.
On the zero-free tail,
\[
 \left|\frac d{dr}\log|f(re^{i\theta})|\right|
 \le \left|\frac{f'(re^{i\theta})}{f(re^{i\theta})}\right|
 \le r^{\rho-1+\delta}.
\]
Integration from a fixed point gives the two-sided bound
\begin{equation}\label{eq.ray.two.sided}
 |\log|f(re^{i\theta})||\le C_{\theta,\delta}r^{\rho+\delta}.
\end{equation}

Fix a direction outside $\mathcal E_G$ and the critical rays of the
frequency hull. Suppose its dominant frequency is $\omega_j$.
Put
\[
 m_j=\max\{k:a_{k,j}\not\equiv0\},\qquad
 b_j=a_{m_j,j},\qquad P_{k,j}=a_{k,j}/b_j\quad(0\le k<m_j).
\]
The polynomial $b_j$ has no zeros beyond a fixed radius.
Dividing the normal form \eqref{standard.eq} by
$e^{\omega_jz}b_j(z)f(z)$ gives, on this ray tail, the exact identity
\begin{equation}\label{eq.expsum.residual}
 \begin{split}
 \epsilon_{j,\theta}(r)
 &:=\sum_{k=0}^{m_j}\frac{a_{k,j}(z)}{b_j(z)}
                       \frac{f^{(k)}(z)}{f(z)}\\
 &=-\sum_{\lambda_\ell\ne\omega_j}
       e^{(\lambda_\ell-\omega_j)z}
       \sum_{k=0}^{n}\frac{a_{k,\ell}(z)}{b_j(z)}
                         \frac{f^{(k)}(z)}{f(z)},
 \qquad z=re^{i\theta}.
 \end{split}
\end{equation}
The rational multipliers are bounded by powers of $r$.
By Lemma \ref{lem.convex.Sj} and \eqref{eq.ray.Gundersen},
\begin{equation}\label{eq.expsum.small}
 |\epsilon_{j,\theta}(r)|\le C_\theta r^{K_\theta}e^{-c_\theta r},
 \qquad c_\theta>0.
\end{equation}
If there is only one frequency group, the right-hand side of
\eqref{eq.expsum.residual} is zero, and the same bound is immediate.
The case $m_j=0$ is impossible, since its left-hand side would be $1$.
Thus $1\le m_j\le n$.

We have proved that the given solution satisfies
\begin{equation}\label{eq.asy.coeffi}
 f^{(m_j)}+P_{m_j-1,j}f^{(m_j-1)}+\cdots+P_{0,j}f
 =\epsilon_{j,\theta}(r)f
\end{equation}
on this ray. In general $m_j$ is smaller than $n$.
The function $\epsilon_{j,\theta}$ is continuous on this ray tail.

Apply Lemma \ref{lem.ray.comparison} to \eqref{eq.asy.coeffi}.
The rational coefficient equation is
$f^{(m_j)}+\sum_{k<m_j}P_{k,j}f^{(k)}=0$.
There are finitely many such equations, and Lemma \ref{lem.ray.comparison}
excludes only finitely many additional directions. Together with $\mathcal E_G$ and the
critical rays, they form a set $\mathcal E_f$ of measure zero.
For every $\theta\notin\mathcal E_f$, we obtain
\begin{equation}\label{eq.log|f|.asym}
 \log|f(re^{i\theta})|
 =\Re G_{j\theta}((re^{i\theta})^{1/p_j})+O(\log r).
\end{equation}
The functions $G_{j\theta}(z^{1/p_j})$ come from a finite list.
Their nonzero leading terms have nonzero real parts on these rays,
by the choice of the additional excluded directions.
The set of remaining directions is dense.
Equations \eqref{eq.log|f|.asym} and \eqref{eq.ray.two.sided}
meet every hypothesis of Lemma \ref{lem.order.crg}.
It follows that $f$ has positive rational order, finite nonzero type,
and \textit{c.r.g}. This proves the theorem.
\end{proof}

\subsection{Proof of Theorem \ref{Thm.main.exp-poly}}
\begin{proof}
Let $f$ be a transcendental entire solution of finite order $\rho$.
Use the same angular estimate \eqref{eq.ray.Gundersen}, with a single
null set $\mathcal E_G$, and hence the same zero-free ray tails and
the two-sided estimate \eqref{eq.ray.two.sided}.
Write the equation in the collected form \eqref{eq.full.exponents}.
By Lemma \ref{lem.full.dominance}, outside a finite set of directions
one $Q_\nu$ has strictly largest real part along each ray tail.
This is the leaf selected by the successive coefficient equations
in Section \ref{subsec.e-p.coeff}.
Indeed, at each step only exponents with the same higher-degree
coefficients are retained. If $Q_\nu-Q_\mu$ has highest term $c z^d$,
its real part on the chosen ray has the sign of $\Re(c e^{id\theta})$
for large $r$. Thus the successive comparisons select the same
exponent as the direct comparison in Lemma \ref{lem.full.dominance}.

Fix such a direction outside $\mathcal E_G$, and put
\[
 m_\nu=\max\{k:a_{k,\nu}\not\equiv0\},\qquad
 b_\nu=a_{m_\nu,\nu},\qquad
 P_{k,\nu}=a_{k,\nu}/b_\nu\quad(0\le k<m_\nu).
\]
The polynomial $b_\nu$ is nonzero and has only finitely many zeros.
Increase the starting radius so that both $b_\nu$ and $f$ are nonzero
on the ray. Division of \eqref{eq.full.exponents} by
$e^{Q_\nu(z)}b_\nu(z)f(z)$ then gives
\begin{equation}\label{eq.exppoly.residual}
 \begin{split}
 \epsilon_{\nu,\theta}(r)
 &:=\sum_{k=0}^{m_\nu}\frac{a_{k,\nu}(z)}{b_\nu(z)}
                         \frac{f^{(k)}(z)}{f(z)}\\
 &=-\sum_{\mu\ne\nu}e^{Q_\mu(z)-Q_\nu(z)}
       \sum_{k=0}^{n}\frac{a_{k,\mu}(z)}{b_\nu(z)}
                            \frac{f^{(k)}(z)}{f(z)},
 \qquad z=re^{i\theta}.
 \end{split}
\end{equation}
The function $\epsilon_{\nu,\theta}$ is continuous on this ray tail.
Each rational multiplier $a_{k,\mu}/b_\nu$ is bounded by a power of $r$,
and \eqref{eq.ray.Gundersen} gives the same kind of bound for
$f^{(k)}/f$. The sum contains only finitely many terms.
Lemma \ref{lem.full.dominance} therefore yields
\begin{equation}\label{eq.exppoly.small}
 |\epsilon_{\nu,\theta}(r)|\le C_\theta r^{K_\theta}e^{-c_\theta r},
 \qquad c_\theta>0.
\end{equation}
If there is just one group, $\epsilon_{\nu,\theta}=0$.
As before, $m_\nu=0$ would give $\epsilon_{\nu,\theta}\equiv1$,
which contradicts this estimate. Thus $m_\nu\ge1$.

Thus
\[
 f^{(m_\nu)}+\sum_{k<m_\nu}P_{k,\nu}f^{(k)}
 =\epsilon_{\nu,\theta}(r)f.
\]
Here $m_\nu\ge1$, each $P_{k,\nu}$ is rational, and $f$ is entire and
eventually nonzero on the ray. Its derivative quotients satisfy
\eqref{eq.ray.Gundersen}, and the continuous error satisfies
\eqref{eq.exppoly.small}. Hence Lemma \ref{lem.ray.comparison} applies.
There are only finitely many fundamental coefficient equations,
one for each exponent $Q_\nu$. For each equation that lemma provides
finitely many functions $G(z^{1/p})$ and excludes finitely many directions.
Their union is still finite and depends only on the original equation.

Let $\mathcal E_f$ be the union of $\mathcal E_G$, the directions
excluded by Lemma \ref{lem.full.dominance}, and these additional
directions. It has measure zero, and its complement is dense.
For every $\theta\notin\mathcal E_f$, we have
\eqref{eq.log|f|.exp-poly.coeffi.} on the whole ray tail.
By the choice of the additional directions, every nonzero selected
function has a highest term with nonzero real part on that ray.
Together with the two-sided bound \eqref{eq.ray.two.sided}, valid for
every $\delta>0$, these are the hypotheses of Lemma \ref{lem.order.crg}.
That lemma proves positive rational order, finite nonzero type, and
\textit{c.r.g.}, as required.
\end{proof}

\subsection{Proof of Corollary \ref{cor.indicator.zeros}}
\begin{proof}
Let $f$ be a transcendental entire solution of the fixed equation
of finite order $\rho=\rho(f)$.
By Theorem \ref{Thm.intro} and Remark \ref{rem.nonmonic},
$f$ has \textit{c.r.g}. Its indicator
\[
 h_f(\theta)=\limsup_{r\to\infty}r^{-\rho}\log|f(re^{i\theta})|
\]
is continuous.

The fundamental coefficient equations depend only on the original
equation. By Lemmas \ref{thm.Wasow} and \ref{lem.ray.comparison},
there is a finite partition of $[0,2\pi)$ into open argument intervals
$I_\nu$, $1\le\nu\le M$, and their endpoints.
On each interval there are finitely many functions
\[
 q_{\nu,j}(z)=G_{\nu,j}(z^{1/p_{\nu,j}}),
 \qquad 1\le j\le N_\nu,
\]
with fixed branches. These functions are independent of $f$.
On almost every ray with $\theta\in I_\nu$, one of them gives
\[
 \log|f(re^{i\theta})|
 =\Re q_{\nu,j}(re^{i\theta})+O(\log r).
\]
Put $d_{\nu,j}=\deg G_{\nu,j}/p_{\nu,j}$, with $d_{\nu,j}=0$
when $G_{\nu,j}=0$.
Lemma \ref{lem.order.crg} shows that $\rho$ is the largest degree
selected on the remaining rays. Hence
\[
 \rho\in\{d_{\nu,j}:d_{\nu,j}>0,\ 1\le\nu\le M,
                                      \ 1\le j\le N_\nu\}.
\]
This is a fixed finite set of positive rational numbers.

Fix one such $\rho$ and an interval $I_\nu$.
For each $q_{\nu,j}$ of degree $\rho$, write its leading term on this
interval as $c_{\nu,j}r^\rho e^{i\rho\theta}$ at $z=re^{i\theta}$.
Set
\[
 C_{\nu,\rho}=\{0\}\cup
               \{c_{\nu,j}:d_{\nu,j}=\rho\}.
\]
A selected function cannot have degree greater than $\rho$.
Division of the ray estimate by $r^\rho$ therefore gives
\begin{equation}\label{eq.indicator.finite.alternatives}
 h_f(\theta)\in
 \{\Re(c e^{i\rho\theta}):c\in C_{\nu,\rho}\},
 \qquad \theta\in I_\nu,
\end{equation}
first on a dense set of directions; a selected function of smaller
degree gives $c=0$.
Continuity of $h_f$ and finiteness of $C_{\nu,\rho}$ extend
\eqref{eq.indicator.finite.alternatives} to every $\theta\in I_\nu$.

Two alternatives corresponding to distinct $c,d\in C_{\nu,\rho}$
are equal precisely at the solutions of
\[
 \Re((c-d)e^{i\rho\theta})=0.
\]
There are only finitely many such points in $I_\nu$.
Divide $I_\nu$ at these points. On each resulting open interval the
alternatives are pairwise distinct, so continuity forces $h_f$ to
equal one fixed alternative throughout that interval.
There are finitely many intervals and finitely many choices on each;
thus only finitely many indicators can occur for this order.
Take the union of the interval endpoints and these equality points
over the finitely many possible orders. These give the directions
$\vartheta_j$ in the statement and prove the first assertion.

For angular intervals crossing $0$, we use the $2\pi$-periodic extension
of $h_f$ and the corresponding translates of the directions $\vartheta_j$.
The zero distribution theorem for such functions
\cite[Chapter III, Section 3, Theorem 3, pp.~152--155]{Levin1} gives
\[
 \lim_{r\to\infty}\frac{n_f(r;\alpha,\beta)}{r^\rho}
 =\frac{h_f'(\beta)-h_f'(\alpha)
       +\rho^2\int_\alpha^\beta h_f(t)\,dt}{2\pi\rho}
\]
when the indicator is differentiable at the endpoints.
On each open interval of the finite partition,
$h_f''+\rho^2h_f=0$.
Integrate on the subintervals between $\alpha$ and $\beta$ and add.
The derivatives at the interior endpoints give
\[
 \begin{aligned}
 h_f'(\beta)-h_f'(\alpha)+\rho^2\int_\alpha^\beta h_f(t)\,dt
 &=\sum_{\alpha<\vartheta_j<\beta}
       \bigl(h_f'(\vartheta_j+0)-h_f'(\vartheta_j-0)\bigr).
 \end{aligned}
\]
This proves \eqref{eq.zero.angular.density}.
Choose $\alpha$ and $\beta$ on opposite sides of a single $\vartheta_j$.
Since the zero count is nonnegative, its limiting density shows that
the corresponding jump is nonnegative.
If the interval contains no $\vartheta_j$, the sum is empty and
$n_f(r;\alpha,\beta)=o(r^\rho)$.
Cover the complement of the prescribed neighbourhoods by finitely many
open angular intervals containing no listed ray.
Applying this estimate to these intervals proves the second assertion.
\end{proof}

\section{An extension via Laurent--Puiseux asymptotics}\label{sec.concluding}
The proof also applies to coefficients with Laurent--Puiseux asymptotic
expansions on finitely many sectors. We state a sufficient condition for
this extension.
Put $a_n\equiv1$ in \eqref{eq.basic}.

Let finitely many rays divide a neighbourhood of infinity into open
sectors $S_\nu$. On each sector, fix a nonvanishing holomorphic function
$W_\nu$, an integer $p_\nu\ge1$, and a branch of $z^{1/p_\nu}$.
Suppose that there are an integer $j_\nu$ and constants $c_{k,\nu,j}$ such
that, for every integer $J\ge j_\nu$,
\begin{equation}\label{eq.puiseux.coefficients}
 \frac{a_k(z)}{W_\nu(z)}
 =\sum_{j=j_\nu}^{J}c_{k,\nu,j}z^{-j/p_\nu}
   +O\bigl(|z|^{-(J+1)/p_\nu}\bigr),
 \qquad 0\le k\le n,
\end{equation}
on every closed subsector of $S_\nu$.
The constants in the estimate and the starting radius may depend on $J$
and the closed subsector. The function $W_\nu$, the branch and all the
coefficients $c_{k,\nu,j}$ are fixed independently of the solution.
We require that, on each sector, at least one of these formal series is
nonzero. These are Laurent--Puiseux asymptotic expansions: finitely many
positive powers are allowed, and the series need not converge.
A zero series means that the corresponding quotient is $O(|z|^{-L})$
for every $L>0$. The condition includes $k=n$, so it also applies to
$1/W_\nu$.

Exponential polynomial coefficients satisfy this condition: group the
terms as in \eqref{eq.full.exponents} and choose $W_\nu$ to be the largest
exponential factor on each of the resulting sectors.

\begin{proposition}\label{prop.puiseux}
Suppose that the coefficients of \eqref{eq.basic} are entire and satisfy
the above conditions. Then every transcendental entire solution
of finite order has positive rational order, finite nonzero type and
\textit{c.r.g}. On almost every ray it also satisfies
\begin{equation}\label{eq.puiseux.ray}
 \log|f(re^{i\theta})|=\Re q_\theta(re^{i\theta})+O(\log r),
\end{equation}
where the $q_\theta$ are chosen from a finite list of polynomials in
fractional powers of $z$, with branches fixed on finitely many argument
intervals. The list depends only on the coefficient expansions.
\end{proposition}
\begin{proof}
Let $f$ be such a solution. On each sector write
\[
 \widehat a_{k,\nu}(z)
 =\sum_{j=j_\nu}^{\infty}c_{k,\nu,j}z^{-j/p_\nu},
 \qquad
 m=\max\{k:\widehat a_{k,\nu}\ne0\}.
\]
Choose a ray outside Gundersen's angular exceptional set. On its tail,
$f$ is nonzero and all $f^{(k)}/f$, $1\le k\le n$, have polynomial bounds,
as in the proof of Theorem \ref{Thm.main}.
If $m=0$, division of the equation by $W_\nu f$ would equate the nonzero
leading term of $a_0/W_\nu$ to a quantity smaller than every negative
power of $r$. This is impossible. Thus $m\ge1$.

The nonzero leading term of $a_m/W_\nu$ shows that $a_m$ has no zeros
sufficiently far out on each closed subsector. Hence, on the chosen ray,
\begin{equation}\label{eq.puiseux.scalar}
 f^{(m)}+\sum_{k=0}^{m-1}\frac{a_k}{a_m}f^{(k)}
 =\epsilon_\theta f,
 \qquad
 \epsilon_\theta
 =-\sum_{k=m+1}^{n}\frac{a_k}{a_m}\frac{f^{(k)}}{f}.
\end{equation}
Each coefficient with $k>m$ has zero expansion. Therefore
$\epsilon_\theta(r)=O(r^{-L})$ for every $L>0$.
The companion matrix $\mathsf A$ of the left-hand side has a
Laurent--Puiseux asymptotic expansion, denoted by $\widehat{\mathsf A}$.

By the Turrittin--Levelt formal reduction theorem, this series admits a
formal reduction at infinity after a finite ramification; see
\cite{BabbittVaradarajan} and \cite[Theorem 4]{Barkatou}.
For an ordinary point in the local variable $x=z^{-1/p_\nu}$,
use its formal power series fundamental matrix.
After replacing $p_\nu$ by a multiple if necessary, this gives a formally
invertible Laurent--Puiseux matrix $\widehat{\mathsf H}$, a constant
matrix $\mathsf G$, and a diagonal matrix $\mathsf Q$ whose entries are
polynomials in $z^{1/p_\nu}$ with zero constant terms. The formal change $Y=\widehat{\mathsf H}V$ reduces
$Y'=\widehat{\mathsf A}Y$ to
$V'=(\mathsf Q'+\mathsf G/z)V$. Hence
\begin{equation}\label{eq.puiseux.formal}
 \widehat{\mathsf A}\widehat{\mathsf H}
 -\widehat{\mathsf H}'
 =\widehat{\mathsf H}\bigl(\mathsf Q'+\mathsf G/z\bigr),
 \qquad \mathsf G\mathsf Q=\mathsf Q\mathsf G.
\end{equation}
The diagonal polynomials and the further choice of fractional powers
are fixed by this formal system.

Let $\mathsf H_N$ be a sufficiently long finite truncation of
$\widehat{\mathsf H}$ and put
\[
 \mathsf D_N
 =\mathsf A\mathsf H_N-\mathsf H_N'
  -\mathsf H_N\bigl(\mathsf Q'+\mathsf G/z\bigr).
\]
By \eqref{eq.puiseux.formal} and the asymptotic expansion of $\mathsf A$,
for any prescribed $L>0$ the truncation can be chosen so that
$\mathsf D_N=O(|z|^{-L})$ on each closed subsector.
Indeed, the coefficients up to any fixed order vanish by
\eqref{eq.puiseux.formal} once sufficiently many terms are retained.
For all sufficiently long truncations the leading powers of the entries
of $\mathsf H_N$ and the leading nonzero term of its determinant are fixed.
Thus $\mathsf H_N$ is invertible far out, and the adjugate formula gives
polynomial bounds for $\mathsf H_N$ and $\mathsf H_N^{-1}$ with an exponent
independent of these truncations. The matrices $z^{\mathsf G}$ and
$z^{-\mathsf G}$ have polynomial bounds as in Lemma \ref{thm.Wasow}.

Set $\mathsf T_N=\mathsf H_Nz^{\mathsf G}$. Direct differentiation gives
\begin{equation}\label{eq.puiseux.truncation}
 \mathsf T_N^{-1}\mathsf A\mathsf T_N
 -\mathsf T_N^{-1}\mathsf T_N'
 =\mathsf Q'+\mathsf S_N,
 \qquad
 \mathsf S_N=z^{-\mathsf G}\mathsf H_N^{-1}\mathsf D_Nz^{\mathsf G}.
\end{equation}
Choose $L$ large enough to absorb the polynomial bounds in this last
product. Then $\mathsf S_N=O(|z|^{-2})$.
Also, $\mathsf T_N$ and its inverse have polynomial bounds.

For $\bm x=(f,f',\ldots,f^{(m-1)})^{\mathsf T}$, use
\eqref{eq.puiseux.scalar} and set $\bm x=\mathsf T_N\bm z$ on the ray.
With $\mathsf E=\bm e_m\bm e_1^{\mathsf T}$, this gives
\[
 \frac{d\bm z}{dr}
 =\left[\frac{d}{dr}\mathsf Q(re^{i\theta})
       +\mathsf R_\theta(r)\right]\bm z,
 \qquad
 \mathsf R_\theta
 =e^{i\theta}\left(\mathsf S_N
       +\mathsf T_N^{-1}\epsilon_\theta\mathsf E\mathsf T_N\right).
\]
Here every matrix is evaluated at $re^{i\theta}$, and
$\int_R^\infty\|\mathsf R_\theta(r)\|\,dr<\infty$.
After excluding the finitely many directions determined by the diagonal
polynomials, Lemma \ref{lem.stein.modify} applies.
The norm comparison in the proof of Lemma \ref{lem.ray.comparison} then
gives \eqref{eq.puiseux.ray}.
There are only finitely many sectors and finitely many choices of branches,
so the list of polynomials is finite.
The two-sided bound for $\log|f|$ follows from the same logarithmic
derivative estimate as in the proof of Theorem \ref{Thm.main}.
Lemma \ref{lem.order.crg} now proves the remaining conclusions.
\end{proof}

We give two coefficient classes satisfying the conditions of
Proposition \ref{prop.puiseux}.

\begin{corollary}\label{cor.analytic.weights}
Suppose that
\begin{equation}\label{eq.analytic.weight.coefficients}
 a_k(z)=E_k(z)+\sum_{\ell=1}^{s}
       \int_{a_\ell}^{b_\ell}w_{k\ell}(t)e^{\alpha_\ell tz}\,dt,
 \qquad 0\le k<n,
\end{equation}
with $a_n\equiv1$.
Here the $E_k$ are exponential polynomials, $s$ is finite,
$a_\ell<b_\ell$ are real numbers, $\alpha_\ell\in\C\setminus\{0\}$,
and each weight $w_{k\ell}$ is a function of $t$, holomorphic
in a complex neighbourhood of $[a_\ell,b_\ell]$.
Then every transcendental entire solution of \eqref{eq.basic}
of finite order satisfies the conclusions of Proposition
\ref{prop.puiseux}.
\end{corollary}
\begin{proof}
Each integral is entire and of exponential type.
For every positive integer $N$, integration by parts gives
\begin{equation}\label{eq.analytic.weight.expansion}
 \begin{aligned}
 \int_a^b w(t)e^{\alpha tz}\,dt
 &=\sum_{j=0}^{N-1}\frac{(-1)^j}{(\alpha z)^{j+1}}
       \bigl(w^{(j)}(b)e^{\alpha bz}
                    -w^{(j)}(a)e^{\alpha az}\bigr)\\
 &\quad+O\!\left(
       \frac{|e^{\alpha az}|+|e^{\alpha bz}|}{|z|^{N+1}}\right),
 \end{aligned}
\end{equation}
on every closed subsector of $\Re(\alpha z)>0$ or
$\Re(\alpha z)<0$. The constant may depend on $N$, $\alpha$, $w$
and the subsector. The remainder follows by integrating the bounded
function $w^{(N)}$ against $e^{\alpha tz}$; the integral of its modulus
is bounded by a constant times
$(|e^{\alpha az}|+|e^{\alpha bz}|)/|z|$.

Put $u=\alpha_\ell t$ and group the resulting segments by their
supporting lines through the origin. On each line, subdivide overlapping
segments and add the analytic weights, including their orientations.
Regard the weight as zero outside the segments.
If all derivatives agree across an endpoint, the adjacent weights
agree near it and can be joined.
After joining such pieces and deleting zero weights, each remaining
nonzero endpoint contributes a nonzero series in at least one component.
Different supporting lines have no common nonzero endpoint.
The exponential factor $1$ is retained because $a_n=1$.

Group the endpoint exponential factors with those of the $E_k$,
combining exponents that differ by constants.
Discard factors whose series are zero in every component.
A polynomial multiplier cannot cancel a nonzero series of negative
powers. On each sector outside finitely many directions, choose
$W_\nu$ to be the largest remaining exponential factor.
Its coefficient series are not all zero, while the other factors are
smaller than every negative power after division by $W_\nu$.
Thus \eqref{eq.puiseux.coefficients} holds with $p_\nu=1$.
Proposition \ref{prop.puiseux} applies.
\end{proof}

A simple special case is given by the \emph{Gaussian weight}
$w(t)=e^{-t^2}$. Taking $s=1$, $a_1=0$, $b_1=1$, $\alpha_1=1$ and
$w_{k1}(t)=d_ke^{-t^2}$ in \eqref{eq.analytic.weight.coefficients} gives
\[
 a_k(z)=E_k(z)+d_kH(z),\qquad 0\le k<n,
 \qquad
 H(z)=\int_0^1 e^{tz-t^2}\,dt,
\]
where the $d_k$ are constants.
Formula \eqref{eq.analytic.weight.expansion} gives
\[
 e^{-z}H(z)\sim\frac{e^{-1}}{z}\quad(\Re z>0),
 \qquad
 H(z)\sim-\frac1z\quad(\Re z<0),
\]
on closed subsectors of the respective half-planes.
The function $H$ is not an exponential polynomial.
Indeed, if it were, so would $e^{-z}H$ be. On a ray in the right
half-plane outside finitely many directions, the latter would be
asymptotic to $P(z)e^{Q(z)}$, where $P$ is a nonzero polynomial
and $Q$ is a polynomial. If $Q$ is nonconstant, its leading term
has nonzero real part on the chosen ray. Neither case is compatible
with $e^{-z}H(z)\sim e^{-1}/z$.
Thus \eqref{eq.analytic.weight.coefficients} includes coefficients
outside the class in Theorem \ref{Thm.intro}.

The next class also allows fractional powers in the coefficient
expansions. The exponent of each integral kernel is a polynomial in
$z$ and $t$ of the form $Q_\ell(z)t^{m_\ell}$.

\begin{corollary}\label{cor.polynomial.kernels}
Suppose that
\begin{equation}\label{eq.polynomial.kernel.coefficients}
 a_k(z)=E_k(z)+\sum_{\ell=1}^{s}
       \int_0^1 w_{k\ell}(t)e^{Q_\ell(z)t^{m_\ell}}\,dt,
 \qquad 0\le k<n,
\end{equation}
with $a_n\equiv1$.
Here the $E_k$ are exponential polynomials, $s$ is finite,
the $Q_\ell$ are nonconstant complex polynomials,
the $m_\ell$ are positive integers, and the weights $w_{k\ell}$
are functions of $t$, holomorphic in complex neighbourhoods of $[0,1]$.
Then every transcendental entire solution of \eqref{eq.basic}
of finite order satisfies the conclusions of Proposition
\ref{prop.puiseux}.
\end{corollary}
\begin{proof}
We first obtain asymptotic expansions associated with the endpoints
$t=0$ and $t=1$ for
\[
 I(\lambda)=\int_0^1 w(t)e^{\lambda t^m}\,dt.
\]
Put
\[
 v(u)=\frac1m u^{1/m-1}w(u^{1/m}),\qquad 0<u\le1.
\]
Then $I(\lambda)=\int_0^1v(u)e^{\lambda u}\,du$.
The function $v$ is analytic near $u=1$ and integrable at $u=0$.
If $\Re\lambda\ge c|\lambda|$ for a fixed $c>0$, split this integral
at a fixed point $0<\delta<1$.
After multiplication by $e^{-\lambda}$, the first part is exponentially
small. Repeated integration by parts in the second part gives
\begin{equation}\label{eq.power.weight.growing}
 e^{-\lambda}I(\lambda)
 =\sum_{j=0}^{N-1}(-1)^jv^{(j)}(1)\lambda^{-j-1}
    +O(|\lambda|^{-N-1}).
\end{equation}

If $\Re\lambda\le-c|\lambda|$, expand $w$ at $t=0$.
Integration of its Taylor polynomial gives
\begin{equation}\label{eq.power.weight.decaying}
 I(\lambda)
 =\sum_{j=0}^{N-1}\frac{w^{(j)}(0)}{j!}
       \frac{\Gamma((j+1)/m)}{m(-\lambda)^{(j+1)/m}}
       +O(|\lambda|^{-(N+1)/m}),
\end{equation}
where the powers of $-\lambda$ have their principal values.
Indeed, the Taylor remainder is bounded by a constant times $t^N$
near zero, and its integral is $O(|\lambda|^{-(N+1)/m})$.
The part away from zero and the tails introduced by extending the
Taylor integrals to infinity are exponentially small.
Both estimates hold uniformly for the indicated ranges of $\lambda$,
for every positive integer $N$.
They are the endpoint forms of Watson's lemma
\cite[Section 2.4(i)]{DLMF}.

Substitute $\lambda=Q_\ell(z)$.
Outside the finitely many directions where the leading term of some
$Q_\ell(re^{i\theta})$ has zero real part, the sign is fixed and
$|\Re Q_\ell(z)|\ge c|Q_\ell(z)|$ on closed subsectors far out.
Formula \eqref{eq.power.weight.growing} gives an exponential factor
$e^{Q_\ell(z)}$ with an expansion in negative integer powers of $z$.
Formula \eqref{eq.power.weight.decaying} gives a Laurent--Puiseux
expansion in $z$; a common denominator is a multiple of all the
$m_\ell$. All the integrals in \eqref{eq.polynomial.kernel.coefficients}
are entire.

Before comparing the exponential factors, combine polynomials $Q_\ell$
that differ by constants. If $Q_\ell=Q+c_\ell$, multiply the corresponding
density $v_{k\ell}(u)$ by $e^{c_\ell u}$ and add the densities.
The sum is analytic near $u=1$ and on $0<u<1$.
If all the coefficients in its asymptotic expansion
\eqref{eq.power.weight.growing} vanish, all its derivatives at $1$
vanish. It is then zero throughout $(0,1)$, so the whole integral
in that component is zero.
Delete any group that is zero in every component.
A nonzero remaining endpoint series has only negative integer powers
of $z$, and cannot be cancelled by a polynomial multiplier in the $E_k$.
Group the factors from the $E_k$ in the same way, absorbing constant
factors into their polynomial multipliers.

The remaining exponential factors, those in the $E_k$ and the factor $1$
form a finite list. On each sector outside finitely many comparison
directions, choose the largest factor as $W_\nu$.
If $W_\nu\ne1$, its coefficient series is nonzero in at least one
component. If $W_\nu=1$, the series for $a_n/W_\nu$ is $1$.
Equations \eqref{eq.power.weight.growing} and
\eqref{eq.power.weight.decaying} therefore verify all the conditions
of Proposition \ref{prop.puiseux}.
\end{proof}

For example, taking $Q(z)=z$ and $m=2$ gives
\[
 J(z)=\int_0^1e^{zt^2}\,dt,
\]
with
\[
 e^{-z}J(z)\sim\frac1{2z}\quad(\Re z>0),
 \qquad
 J(z)\sim\frac{\sqrt\pi}{2(-z)^{1/2}}\quad(\Re z<0),
\]
on closed subsectors of the respective half-planes.
Thus fractional powers already occur for an entire coefficient
of exponential type. The choices $Q(z)=\alpha z$ and
$Q(z)=\alpha z^m$ include the kernels $e^{\alpha zt^m}$ and
$e^{\alpha(tz)^m}$, respectively.

\section*{Acknowledgements}
The author thanks Zhi-Tao Wen for providing helpful references, and
Walter Bergweiler for his comments and suggestions on an earlier version
of the manuscript during the author's visit to Kiel University.
The author acknowledges support from the China Scholarship Council
(CSC No.~202508440485) during the revision of this paper.

\section*{AI assistance and formal verification}
OpenAI's ChatGPT and Codex were used to assist with literature searches,
to develop and check parts of the arguments, to revise the exposition
and LaTeX source, and to develop and debug the Lean formalizations.
The author is responsible for the mathematical content, references and
final presentation.

All numbered theorem, proposition, and corollary statements in this
manuscript have been formalized in Lean~4. The fixed source version and
verification report are cited in \cite{LiLean}. All 320 modules passed the
recorded verification, which audited 1,734 distinct declarations; the
sources contain no \texttt{sorry}.

	\bigskip
	\noindent
	{\itshape
Xingyu Li\\
Department of Mathematics\\
The University of Manchester\\
Oxford Road, Manchester, M13 9PL, United Kingdom\\
\href{mailto:xingyu.li@manchester.ac.uk}{xingyu.li@manchester.ac.uk}\par
}

\end{document}